\documentclass[a4paper]{amsart}

\usepackage{euscript,amsrefs,amssymb,mathrsfs,lineno}

\newtheorem{theorem}{Theorem}[section]
\newtheorem{lemma}[theorem]{Lemma}

\newtheorem{question}{Question}

\newtheorem{proposition}[theorem]{Proposition}

\newtheorem{mainthm}{Theorem}
\newtheorem{maincor}[mainthm]{Corollary}

\newcommand\N{{\mathbb N}}
\newcommand\K{{\mathbb K}}
\DeclareMathOperator{\GKdim}{GKdim}

\begin{document}

\title{Images of Golod-Shafarevich algebras with small growth}

\author{Agata Smoktunowicz}
\address{Maxwell Institute for Mathematical Sciences\\
  School of Mathematics, University of Edinburgh\\
  James Clerk Maxwell Building, King's Buildings, Mayfield Road\\
  Edinburgh EH9 3JZ, Scotland, UK}
\email{A.Smoktunowicz@ed.ac.uk}

\author{Laurent Bartholdi}
\address{Georg-August Universit\"at zu G\"ottingen}
\email{Laurent.Bartholdi@gmail.com}

\thanks{The first author was supported by was supported by
  EPSRC grant \#EP/D071674/1, and the second author was supported by  the Courant research Centre
  ``Higher Order Structures''}

\date{August 21, 2011; revised February 7, 2012 and August 25, 2012}

\begin{abstract}
  We show that Golod-Shafarevich algebras can be homomorphically
  mapped onto infinite dimensional algebras with polynomial growth
  when mild assumptions about the number of relations of given degrees
  are introduced. This answers a question by Zelmanov.

  In the case that these algebras are finitely presented, we show that
  they can be mapped onto infinite dimensional algebras with at most
  quadratic growth.

  We then use an elementary construction to show that any sufficiently
  regular function $\succsim n^{\log n}$ may occur as the growth
  function of an algebra.
\end{abstract}


\maketitle


\section*{Introduction}

\begin{quote}
  \raggedright ``Is every finitely generated torsion group finite?''\par
  \raggedleft (the \emph{Burnside} problem)

  \raggedright ``Is every finitely generated algebraic algebra finite dimensional?''\par
  \raggedleft (the \emph{Kurosh} problem)
\end{quote}

The seminal work of Golod and
Shafarevich~\cites{golod:nil,golod-s:orig} in~1964 showed that the
answer to these famous problems is negative. Their method entailed the
construction of an infinite dimensional, finitely generated nil graded
algebra $R$ by carefully adding relators; the elements of the form
$1+n$ for $n$ in the generating set of $R$ generate an infinite
torsion group.

Their construction is quite flexible, and has been generalized in
various directions to obtain more information on the \emph{growth} of
$R$, which we now define.

Let $R$ be an associative algebra generated by a finite dimensional
subspace $S$. Then the \emph{growth} of $R$ is the function
$v(n)=\dim(1+S+S^2+\cdots+S^n)$. This depends upon $S$, but only mildly: if
$S'$ be another generating subspace for $R$, then the corresponding
growth function $v'$ is related to $v$ by the inequalities
\[v(n)\le v'(Cn),\qquad v'(n)\le v(Cn)\] for some constant $C$. We
write $v\precsim v'$ if one, and $v\sim v'$ if both of the
inequalities above are satisfied; then the equivalence class of $v$ is
independent of $S$. (Note that~\cite{krause-l:gkdim}*{pp. 5--6} define
a slightly weaker relation, based on the inequalities $v(n)\le Cv'(Cn)$
and $v'(n)\le Cv(Cn)$; our results hold for both.)

The algebra $R$ has \emph{polynomial} growth if $v(n)\precsim n^d$ for
some $d$; the infimal $d$ such that $v(n)\precsim n^d$ is the
\emph{Gelfand-Kirillov} dimension of $R$. $R$ has
\emph{superpolynomial} growth if no such $d$ exists, it has
\emph{exponential} growth if $v(n)\sim 2^n$ and \emph{subexponential}
growth otherwise.

The groups and the algebras constructed by the Golod-Shafarevich
method have exponential growth. Much later, Gromov~\cite{gromov:nilpotent}
proved that under the assumption that the group has polynomial growth,
the answer to the Burnside Problem is positive. In fact, he proved
that a finitely generated group with polynomial growth has a nilpotent
normal subgroup of finite index. As a consequence, if a
finitely generated group has polynomial growth and each element has
finite order then the group is finite.

Zelmanov asked in~\cite{zelmanov:openpbalgebras}*{Problem~5} whether
Golod-Shafarevich algebras always have infinite dimensional
homomorphic images with polynomial growth. It was shown
in~\cite{smoktunowicz:GKgeneric} that this is not the case.  The same was
shown for Golod-Shafarevich groups by Ershov~\cite{ershov:goshat} ---
there exist Golod-Shafarevich groups without infinite images of
polynomial growth; indeed there exist Golod-Shafarevich groups
satisfying Kazhdan's property~(T).

\noindent Our first main result is the following
\begin{mainthm}[Golod-Shafarevich algebras]\label{thm:gosha}
  Let $\K$ be an algebraically closed field, and let $A=\K\langle x,
  y\rangle$ be the free noncommutative algebra generated (in degree
  one) by elements $x,y$. Let $d\ge 50$ be given; let a sequence
  $(D(k))_{k\in\N}$ of subspaces of $A$ be given, with $D(k)$
  homogeneous of degree $k$, such that
  \begin{enumerate}\renewcommand\theenumi{\ref{thm:gosha}.\arabic{enumi}}
  \item $D(k)=0$ if $k/\log k<18d$;\label{gosha:1}
  \item $D(k)=0$ if $2^n-2^{n-3}<k\le2^n+2^{n-1}$ for some $n\in\N$;
  \item $D(k)=0$ if there exist $j,n\in\N$ such that $D(j)\neq0$
    and $j<2^n<k<\max\{j^{1000},12dnj\}$;\label{gosha:3}
  \item $\dim D(k)\le k^d$ for all $k\in\N$.
  \end{enumerate}

  Then $A/\langle D(k)\colon k\in\N\rangle$ can be homomorphically
  mapped onto an infinite dimensional algebra with Gelfand-Kirillov
  dimension at most $45d$.
\end{mainthm}

\begin{mainthm}[finitely presented Golod-Shafarevich algebras]\label{cor:gosha}
  With the same assumptions and notations as in
  Theorem~\ref{thm:gosha}, if we assume $D(k) = 0$ for almost all
  $k$, then $A/\langle D(k)\colon k\in\N\rangle$ can be mapped onto an
  infinite dimensional algebra with at most quadratic growth.
\end{mainthm}

We recall the definition of Golod-Shafarevich algebras;
see~\cite{zelmanov:openpbalgebras}. Let $A$ be the free associative
algebra on the set of free generators $X=\{x_1, x_2, \ldots ,x_m\}$
over a field $\K$. It is graded, with generators $x_i$ of degree
$1$. Let $\mathscr R$ be a graded subspace of $A$, without constant
term, and write $\mathscr R(i)$ its homogeneous component of degree
$i$. The formal series $H_{\mathscr R}(t)=\sum_{i=1}^{\infty }\dim
\mathscr R(i) t^i$ is called the \emph{Hilbert series} of $\mathscr
R$.  The algebra $Q=\langle X\mid \mathscr R\rangle$ presented by
generators $X$ and relators $\mathscr R$ is graded; call $Q(i)$ its
degree-$i$ component, and $H_Q(t)=\sum_{i=0}^{\infty}\dim Q(i) t^i$
its Hilbert series. Golod and Shafarevich proved
in~\cite{golod-s:orig} that $H_Q(t)(1-mt+H_{\mathscr R}(t))\ge 1$
holds co\"efficient-wise. Therefore, if there exists a number
$t_0\in(0,1)$ such that $H_{\mathscr R}(t)$ converges at $t_0$ and
$1-mt_0+H_{\mathscr R}(t_0)<0$, then $Q$ is infinite dimensional. An
algebra admitting such a presentation is called a
\emph{Golod-Shafarevich algebra}.

Observe that the conditions `$D(k)\le k^d$ for all $k\in\N$' combined
with `$D(k)=0$ for all $k<\ell(d)$', where $\ell(d)$ is sufficiently
large, together imply that $A/\langle D(k)\colon k\in\N\rangle$ is a
Golod-Shafarevich presentation. We therefore obtain an affirmative
answer to the following
\begin{question}[Zelmanov, private communication]
  Is it true that every Golod-Shafarevich algebra has an
  infinite dimensional homomorphic image of finite Gelfand-Kirillov
  dimension, under the additional assumption that the number of
  relations is small and the relations have sparse degrees?
\end{question}

We remark that Alexander Young obtained related results, but for
special types of ideals with repeated patterns, called
\emph{regimented ideals}. For example a regimented ideal generated by
single $f\in A$ is of type $\bigcap_{1\le i\le\deg f}\sum_{k\in(\deg
  f)\N} A(k)fA$. The ideal $\langle D(k)\colon k\in\N\rangle$ that we
consider does not suffer from such restrictions.

More generally, we may wish to construct algebras of prescribed
growth in which a predetermined set of relations have already been
imposed. In this sense, we are able to achieve finite Gelfand-Kirillov
dimension when the relations are in appropriately separated degrees.

According to Gromov's result mentioned above, if a group has
polynomial growth then its growth function is $\sim n^d$ for an
integer $d$. Grigorchuk showed in~\cite{grigorchuk:growth} that there
exist semigroups of growth strictly between polynomial and
exponential; closely related
examples~\cite{bartholdi-erschler:permutational} have growth
$\exp(n^\alpha)$ for various $\alpha$, accumulating to $1$. One of the
tantalizing open problems is the existence of groups of intermediate
growth strictly between polynomial and $\exp(n^{1/2})$.

Which functions are the growth function of an associative algebra? An
obvious restriction is that the growth function must be
submultiplicative (since any $(m+n)$-fold product of generators can be
factored as an $m$-fold product times an $n$-fold product). For every
real number $\alpha\ge2$, there exists a finitely generated algebra
with Gelfand-Kirillov dimension $\alpha$; see~\cite{krause-l:gkdim}.
We address the question of constructing algebras of superpolynomial
growth:
\begin{mainthm}\label{thm:arbitrary}
  Let $f\colon\N\to\N$ be submultiplicative and increasing, that is,
  $f(m+n)\le f(m)f(n)$ for all $m,n$, and $f(n+1)\ge f(n)$.  Then
  there exists a finitely generated algebra $B$ whose growth function
  $v(n)$ satisfies
  \[f(2^n)\le \dim B(2^n)\le 2^{2n+3}f(2^{n+1}).\]
  Furthermore, $B$ may be chosen to be a monomial algebra.
\end{mainthm}

\begin{maincor}[Many growth functions]\label{cor:many}
  Let $f\colon\N\to\N$ be submultiplicative, increasing, and such that
  $f(Cn)\ge nf(n)$ for some $C>0$ and all $n\in\N$. Then there exists
  an associative algebra with growth $\sim f$.
\end{maincor}
Note that the hypotheses are satisfied by any sufficiently regular
function that grows at least as fast as $n^{\log n}$. The results
in~\cite{trofimov:growth} therefore hold for a very large class of
growth functions.

It remains open whether arbitrary functions between
polynomial and $n^{\log n}$ can be realized as the growth of an
algebra.

The proofs of Theorems~\ref{thm:gosha} and~\ref{thm:arbitrary} are
closely related: in both cases, good control is achieved on certain
subspaces of monomials in degree a power of two, which represent (in a
strong sense) linearly independent elements in the algebra we are
about to construct; then the properties of these subspaces are carried
over to all degrees.

Our construction is elementary, and bears resemblance to Zelmanov's
construction of a prime algebra with a nonzero locally nilpotent
ideal~\cite{zelmanov:prime}. We do not know if our algebras are prime.

\subsection{Notation}
In what follows, $\K$ is an algebraically closed field (this
is needed because we will apply Hilbert's Nullstellensatz). By $A$ we
denote the free $\K$-algebra in two non-commuting indeterminates $x$
and $y$. The set of monomials in $\{x,y\}$ is denoted by $M$ and, for
each $k\geq 0$, its subset of monomials of degree $k$ is denoted by
$M(k)$.  Thus, $M(0)=\{1\}$ and for $k\geq1$ the elements in $M(k)$
are of the form $x_1\cdots x_k$ with all $x_i\in \{x, y\}$. The span
of $M(k)$ in $A$ is denoted by $A(k)$; its elements are called
\emph{homogeneous polynomials of degree $k$}. More generally, for any
subset $X$ of $A$, we denote by $X(k)$ its subset of homogeneous
elements of degree $k$.

The \emph{degree} $\deg f$ of an element $f \in A$ is the least
$k\ge0$ such that $f \in A(0) + \cdots + A(k)$. Any $f\in A$ can be
uniquely written in the form $f=f_0+f_1+\cdots+f_k$ with each $f_i\in
A(i)$. The elements $f_i$ are the \emph{homogeneous components} of
$f$.  A (right, left, two-sided) ideal of $A$ is \emph{homogeneous} if
it is spanned by its elements' homogeneous components.  If $V$ is a
linear space over $\K$, we denote by $\dim V$ the dimension of $V$
over $\K$.  The Gelfand-Kirillov dimension of an algebra $R$ is
written $\GKdim(R)$. For elementary properties of Gelfand-Kirillov
dimension we refer to~\cite{krause-l:gkdim}.

For any real number $x$, define $\lfloor x \rfloor$ as the largest
integer at most $x$, and $\lceil x\rceil$ as the smallest integer at
least $x$. All logarithms are in base $2$.

\subsection{Sketch of the proof of Theorem~\ref{thm:gosha}}
Its proof is an intricate induction, which we broadly explain here. We
are given a collection of subspaces $D(k)\subseteq A(k)$, namely a
graded subspace $D$ of $A$; these are relations that will hold in the
algebra we are about to construct.

We construct auxiliary graded subspaces $U,V,F,\mathscr E$ of $A$; the
spaces $U,V,F$ vanish except in dimensions a power of two. They have
the following informal meaning:
\begin{itemize}
\item elements of $F$ are chunks, with length a power of two, of the
  relations we want to impose; in degrees a power of two, we have
  $ADA\subseteq FA+AF$;
\item elements of $U$ are chunks, with length a power of two, of the
  ideal generated by $F$. They are complemented, in $A$, by the spaces
  $V$. We therefore have, again in degrees a power of two, $A=U\oplus
  V$ and $F\subseteq U$ and $U\subseteq AU+UA$ and $V\subseteq VV$;
\item Roughly speaking, $\mathscr E$ is the largest ideal contained (in degrees a power
  of two) in $AU+UA$; so we have $A\mathscr EA\subseteq AU+UA$.
\end{itemize}

The construction, based on looking at chunks of relations in degrees a
power of two, makes the combinatorics tractable; on the other hand, it
does not allow very tight control on the growth of the resulting
algebra $A/\mathscr E$. It is nevertheless sufficient to control the
Gelfand-Kirillov dimension of $A/\mathscr E$, which has a basis made of
products of appropriate monomials from a basis of $V$.

The following is a more precise road map of where the various spaces
$U,V,F,\mathscr E$ are constructed, and where their various properties
are proven:
\begin{itemize}
\item in~\S\ref{ss:F}, subspaces $F(2^n)$ of $A(2^n)$ are constructed,
  depending on splittings $U(2^m)\oplus V(2^m)=A(2^m)$ for $m<n$ and
  on the subspaces $D(k)$ for $2^n+2^{n-1}\le k\le
  2^n+2^{n-1}+2^{n-2}$. Roughly speaking, for $k\ge 2^n$, elements of
  degree $2^{n+1}$ in $AD(k)A$ are contained in
  $F(2^n)A(2^n)+A(2^n)F(2^n)$.
\item in~\S\ref{ss:UV}, subspaces $U(2^n),V(2^n)\subseteq A(2^n)$ are
  constructed, depending on $U(2^m),V(2^m)$ for $m<n$ and on
  $F(2^n)$. This part relies heavily on previous results
  from~\cite{lenagan-s:nillie}. Among other properties, they satisfy
  $V(2^{n-1})^2\subseteq V(2^n)$.
\item still in~\S\ref{ss:UV}, for all $k\in\{2^{n-1},\dots,2^n-1\}$,
  we set
  \[\mathscr E(k)=\{r\in A(k)\mid ArA\cap A(2^{n+1})\subseteq U(2^n)A(2^n)+A(2^n)U(2^n)\}\]
  and $\mathscr E=\bigoplus_k\mathscr E(k)$. The desired quotient is
  then $A/\mathscr E$.
\item in~\S\ref{ss:GK}, we bound the Gelfand-Kirillov dimension of
  $A/\mathscr E$.
\end{itemize}

Let us now give a concrete example how the construction works using
simple sets $V(2^n)$, $U(2^n)$. We take $A=\K\langle x,y\rangle$, and
for all $n\in\N$, set $V(2^n)=\K x^{2^n}+\K y^{2^n}$, and let
$F(2^n)=U(2^n)$ be spanned by all monomials of degree $2^n$ which
contain both $x$ and $y$. By the definition of $\mathscr E$, we see
that $x^{2^n}y^{2^n}\notin\mathscr E$ for all $n\in\N$, because
$x^{2^n}y^{2^n}\notin U(2^n)A(2^n)+A(2^n)U(2^n)$; similarly
$y^{2^n}x^{2^n}\notin\mathscr E$.  Let $\langle x \rangle, \langle y
\rangle$, denote the ideals generated respectively by $x$ and $y$ in
$A$. Observe that, by definition of $A$, both $\langle x \rangle
\langle y \rangle \langle x \rangle$ and $\langle y \rangle \langle x
\rangle \langle y \rangle$ are contained in $\mathscr E$.  It follows
that $A/\mathscr E=A/(\langle x \rangle \langle y \rangle \langle x
\rangle+\langle y \rangle \langle x \rangle \langle y \rangle)$ is a
complete description of this example. The situation when sets $U(2^n)$
are not spanned by monomials is usually more complicated.

The proof of Theorem~\ref{thm:gosha} and Theorem~\ref{cor:gosha}  is concluded
in~\S\ref{ss:proof}. We prove Theorem~\ref{thm:arbitrary} and its
corollary in~\S\ref{ss:arbitrary}.

\section{Constructing $U(2^n),V(2^n)$ from $F(2^n)$}\label{ss:UV}
We begin with a modification of
\cite{lenagan-s:nillie}*{Theorem~3}. In~\S\ref{ss:GK}, we will find
bounds on the Gelfand-Kirillov dimension of algebras constructed using
the subspaces $U(2^n)$, $V(2^n)$ described in this~\S.  We will
use Conditions~\eqref{prop:2} and~\eqref{prop:3} of
Theorem~\ref{8props} in~\S\ref{ss:GK}--\ref{ss:F}.

We assume throughout the paper that subspaces $D(k)$ of $A(k)$ and an
integer $d$ are fixed, and define
\begin{align}
  e(n)&=\lfloor \log(5dn)\rfloor,\label{def:e}\\
  Y&=\{n\in\N\colon D(k)\neq0\text{ with }2^n+2^{n-1}\le k\le 2^n+2^{n-1}+2^{n-2}\},\label{def:Y}\\
  S&=\bigsqcup_{n\in Y}\{n-e(n)-1,\dots,n-1\}.\label{def:S}
\end{align}

\begin{lemma}
  Under the assumptions of Theorem~\ref{thm:gosha}, the union defining
  $S$ above is indeed disjoint and contained in $\N$.
\end{lemma}
\begin{proof}
  Assume first that $n-e(n)-1<0$ for some $n\in Y$; so
  $n-1<\log(5dn)$, and $2^n<10dn$.  Then there is $k\in\N$ with
  $D(k)\neq0$ and $\frac32 2^n\le k\le \frac74 2^n$. Therefore,
  \[k<\tfrac74 10dn,\text{ and }k/\log k<18d,\]
  contradicting~\eqref{gosha:1}.

  Assume that the union is not disjoint; then there are $m<n$ in $Y$
  such that $n-e(n)-1\le m-1$; so there are $j<k$ with $D(j)\neq0$,
  $D(k)\neq0$ and $\frac32 2^m\le j\le \frac74 2^m$ and $\frac32
  2^n\le k\le \frac74 2^n$. Therefore,
  \[j<2^n<k\le 2^{e(n)}\tfrac74 2^m\le 10dn\tfrac76j\le 12dnj,\]
  contradicting~\eqref{gosha:3}.
\end{proof}

The following theorem encapsulates the construction of homogeneous
spaces $U(2^n),V(2^n)$ out of appropriate homogeneous spaces
$F(2^n)$. In its first part, it describes properties satisfied by
these spaces, assuming that all $F(2^n)$ are given. The construction
is in fact inductive, and the second part claims that the construction
of $U(2^n),V(2^n)$ only depends on the $F(2^m)$ for $m\le n$ and the
$U(2^m),V(2^m)$ for $m<n$; once the $U(2^m),V(2^m)_{m<n}$ have been
constructed inductively, it is not necessary to modify them.

\begin{theorem}\label{8props}
  Let $Y$ be as in~\eqref{def:Y} and let $S$ be as
  in~\eqref{def:S}. Let an integer $n$ be given.  Suppose that, for
  every $m\le n$, we are given a subspace $F(2^m)\subseteq A(2^m)$
  with $\dim F(2^m)\le (2^{2^{e(m)}})^2-2$ and that, for every $m<n$,
  we are given subspaces $U(2^m),V(2^m)$ of $A(2^m)$ with
  \begin{enumerate}\renewcommand\theenumi{\ref{8props}.\roman{enumi}}
  \item $\dim V(2^m)=2$ if $m\notin S$;\label{prop:1}
  \item $\dim V(2^{m-e(m)-1+j})=2^{2^j}$ for all $m\in Y$ and all $0\le
    j\le e(m)$;\label{prop:2}
  \item $V(2^m)$ is spanned by monomials;\label{prop:3}
  \item $F(2^m)\subseteq U(2^m)$ for every $m\in Y$, and $F(2^{m})=0$
    for every $m\notin Y$;\label{prop:4}
  \item $V(2^m)\oplus U(2^m)=A(2^m)$;\label{prop:5}
  \item $A(2^{m-1})U(2^{m-1})+U(2^{m-1})A(2^{m-1})\subseteq U(2^m)$;\label{prop:6}
  \item $V(2^m)\subseteq V(2^{m-1})V(2^{m-1})$.\label{prop:7}\label{prop:last}
  \end{enumerate}

  Then there exist subspaces $U(2^n),V(2^n)$ of $A(2^n)$ such that the
  extended collection $U(2^m),V(2^m)_{m\le n}$ still satisfies
  Conditions~(\ref{prop:1}--\ref{prop:last}).
\end{theorem}

The proof is very similar to the proof of
\cite{lenagan-s:nillie}*{Theorem~3}. We nevertheless include it for
completeness.

\begin{proof}
  We begin the inductive construction by setting $V(2^0)=\K x+\K y$
  and $U(2^0)=0$. Assume that we have defined $V(2^m)$ and $U(2^m)$
  for all $m\leq n$ in such a way that
  Conditions~(\ref{prop:1}--\ref{prop:last}) hold for all $m\leq
  n$. Define then $V(2^{n+1})$ and $U(2^{n+1})$ in the following
  way. We have three cases:
  \begin{itemize}
  \item[1.] $n\in S$ and $n+1\in S$;
  \item[2.] $n\notin S$;
  \item[3.] $n\in S$ and $n+1\notin S$.
  \end{itemize}

  \newcommand{\listlabel}[1]{ \bfseries #1.}
  \begin{list}{}{\itemindent=0em\leftmargin=0em\parsep=1ex
      \let\makelabel\listlabel}
  \item[Case 1]
  If $n\in S$ and $n+1\in S$, define
  \[U(2^{n+1}) = A(2^n)U(2^n) + U(2^n)A(2^n)\text{ and }
  V(2^{n+1})=V(2^n)V(2^n).
  \]
  Conditions~(\ref{prop:6},\ref{prop:7}) certainly hold. Since
  Conditions~(\ref{prop:3},\ref{prop:5}) hold for $U(2^n)$ and
  $V(2^n)$, they hold for $U(2^{n+1})$ and $V(2^{n+1})$ as
  well. Moreover, $\dim V(2^n)=(\dim V(2^n))^2$, inductively
  satisfying Condition~(\ref{prop:2}). Condition~(\ref{prop:4}) holds
  trivially since $n+1\notin Y$.

  \item[Case 2]
  Suppose that $n\notin S$. Then $\dim V(2^n)=2$, and is
  generated by monomials, by the inductive hypothesis. Let $m_1, m_2$
  be two monomials that generate $V(2^n)$. Then $V(2^n)V(2^n)=\K
  m_1m_1 + \K m_1m_2 + \K m_2m_1 + \K m_2m_2$. Set
  \[V(2^{n+1})=\K m_1m_1 + \K m_1m_2,
  \]
  so that Conditions~(\ref{prop:1},\ref{prop:3},\ref{prop:7}) hold, and
  \[U(2^{n+1}) = A(2^n)U(2^n) + U(2^n)A(2^n) + m_2 V(2^n).
  \]
  Using this definition, Condition~\eqref{prop:6} holds and
  \begin{align*}
    A(2^{n+1}) &= A(2^n)A(2^n)\\
    &=U(2^n)U(2^n) \oplus U(2^n)V(2^n) \oplus V(2^n)U(2^n)
    \oplus m_1 V(2^n) \oplus m_2 V(2^n)\\
    &=U(2^{n+1}) \oplus V(2^{n+1}),
  \end{align*}
  so Condition~\eqref{prop:5} also holds. Condition~\eqref{prop:4}
  holds trivially since $n+1\notin Y$. Condition~\eqref{prop:2} holds
  because if $n+1\notin S$ then $\dim V(2^{n+1})=2$ as required, and
  if $n+1\in S$ then $n+1=m-e(m)-1$ for some $m\in Y$, using $n\notin
  S$; and then $\dim V(2^{n+1})=\dim V(2^{m-e(m)-1+0})=2^{2^{0}}=2$ as
  required.

  \item[Case 3]
  Suppose that $n\in S$ while $n+1\notin S$. Then
  $n+1\in Y$. Induction using Condition~\eqref{prop:2} gives $\dim
  V(2^n)= \dim V(2^{n+1-e(n+1)-1+e(n+1)})=2^{2^{e(n+1)}}$, and $\dim
  V(2^n)V(2^n)=(2^{2^{e(n+1)}})^{2}$. Induction using
  Condition~\eqref{prop:5} gives
  \[A(2^{n+1}) = U(2^n)U(2^n) \oplus U(2^n)V(2^n)
  \oplus V(2^n)U(2^n) \oplus    V(2^n)V(2^n).
  \]

  Let $\{f_{1}, \dots ,f_{s}\}$ be a basis of $F(2^{n+1})$, for some
  $f_{1},\dots ,f_{s}\in A(2^{n+1})$ and $s\leq
  (2^{2^{e(n+1)}})^2-2$. Each $f_j$ can be uniquely decomposed into
  $\bar{f}_j+g_j$ with $\bar {f}_j \in V(2^n)V(2^n)$ and $g_{j}\in
  V(2^n)U(2^n) +U(2^n)U(2^n)+U(2^n)V(2^n)$. Let $P$ the subspace
  spanned by $\bar{f}_1,\dots, \bar{f}_s$.

  Since $\dim_\K P \leq s = \dim F(2^{n+1}) \leq (2^{2^{e(n+1)}})^2-2=
  V(2^n)V(2^n) - 2$, there exist at least two monomials $m_1, m_2 \in
  V(2^n)V(2^n)$ such that the space $\K m_1\oplus\K m_2$ is disjoint
  from $P$.  Define
  \[V(2^{n+1})=\K m_1\oplus\K m_2;
  \]
  this space satisfies
  Conditions~(\ref{prop:1},\ref{prop:3},\ref{prop:7}).
  Since $P$ is disjoint from $\K m_1 + \K m_2$, there exists a space
  $Q \supseteq P$ such that $V(2^n)V(2^n) = Q \oplus (\K m_1+\K m_2)$.
  Set
  \[U(2^{n+1}) = U(2^n)U(2^n) + U(2^n)V(2^n) + V(2^n)U(2^n) + Q.\]
  This immediately satisfies Conditions~(\ref{prop:5},\ref{prop:6}).
  Since each polynomial $f_j = g_j + \bar{f}_j$ belongs to
  $U(2^{n+1})$, the set $U(2^{n+1})$ satisfies
  Condition~\eqref{prop:4} as well. Condition~\ref{prop:2} is
  satisfied trivially since $n+1\in Y$, hence $n+1\notin S$.
  \end{list}
\end{proof}

Let subspaces $U(2^n)$, $V(2^n)$ satisfying
Conditions~(\ref{prop:1}--\ref{prop:last}) of Theorem~\ref{8props} be given,
for all $n\in\N$. We then define a graded subspace $\mathscr E$ of $A$
by constructing its homogeneous components $\mathscr E(k)$ as
follows. Given $k\in\N$, let $n\in\N$ be such that $2^{n-1}\le
k<2^n$. Then $r\in\mathscr E(k)$ precisely if, for all
$j\in\{0,\dots,2^{n+1}-k\}$, we have $A(j)rA(2^{n+1}-j-k)\subseteq
U(2^n)A(2^n)+A(2^n)U(2^n)$. More compactly,
\begin{equation}\label{def:E}
  \mathscr E(k)=\{r\in A(k)\mid ArA\cap A(2^{n+1})\subseteq U(2^n)A(2^n)+A(2^n)U(2^n)\};
\end{equation}
and then $\mathscr E=\bigoplus_{k\in\N}\mathscr E(k)$. We quote
\begin{lemma}[\cite{lenagan-s:nillie}*{Theorem~5}]
  The set $\mathscr E$ is an ideal in $A$.
\end{lemma}

We now extend the definition of $U(2^n),V(2^n)$ to dimensions that are
not powers of $2$. The sets~(\ref{def:U<}--\ref{def:V>}) are called
respectively $S,W,R,Q$ in~\cite{lenagan-s:nillie}*{\S4}.

Let $k\in\N$ be given. Write it as a sum of increasing powers of $2$,
namely $k=\sum_{i=1}^t 2^{p_i}$ with $0\le p_1 < p_2 < \ldots <
p_t$. Set then
\begin{align}
  U^<(k) &= \sum_{i=0}^t A(2^{p_1}+\cdots+2^{p_{i-1}})U(2^{p_i})A(2^{p_{i+1}}+\cdots+2^{p_t}),\label{def:U<}\\
  V^<(k) &= V(2^{p_1})\cdots V(2^{p_t}),\label{def:V<}\\
  U^>(k) &= \sum_{i=0}^t A(2^{p_t}+\cdots+2^{p_{i+1}})U(2^{p_i})A(2^{p_{i-1}}+\cdots+2^{p_1}),\\
  V^>(k) &= V(2^{p_t})\cdots V(2^{p_1}).\label{def:V>}
\end{align}

\begin{lemma}[\cite{lenagan-s:nillie}*{pp.~993--994}]\label{lem:A=U+V}
  For all $k\in\N$ we have $A(k)=U^<(k)\oplus V^<(k)=U^>(k)\oplus
  V^>(k)$.

  For all $k,\ell\in\N$ we have $A(k)U^<(\ell)\subseteq U^<(k+\ell)$ and $U^>(k)A(\ell)\subseteq U^>(k+\ell)$.
\end{lemma}

\noindent These sets are useful to estimate the dimension of $A/\mathscr E$:
\begin{proposition}[\cite{lenagan-s:nillie}*{Theorem~11}, \cite{lenagan-s-y:nil}*{Theorem~5.2}]\label{prop:A/Egrowth}
  For every $k\in\N$ we have
  \[\dim A(k)/\mathscr E(k)\le \sum_{j=0}^k\dim V^<(k-j)\dim V^>(j).\]
\end{proposition}

We remark that, although in the proofs above we copied results
from~\cite{lenagan-s:nillie}, we do not need to assume that the base
field $\K$ be countable. Indeed, this assumption was only used
in~\cite{lenagan-s:nillie} to enumerate powers of elements of $A$ in
order to construct a nil algebra; and we do not need to do this here.

\section{The Gelfand-Kirillov dimension of $A/\mathscr E$}\label{ss:GK}
In this~\S, we estimate the Gelfand-Kirillov dimension of the algebra
$A/\mathscr E$ that was constructed in the previous~\S. To lighten
notation, we write $[X]=\dim X$ for the dimension of a
subspace $X\subseteq A$.

We start with a lemma about the dimensions $V^>(k)$ and $V^<(k)$,
continuing on the notation of~\S\ref{ss:UV}.

\begin{lemma}\label{lem:3.1}
  Let $\alpha$ be a natural number, with binary decomposition $\alpha
  = 2^{p_1} + \cdots +2^{p_{t}}$. Suppose $p_i\notin S$ for all
  $i=1,\ldots,t$. Then $[V^>(\alpha)]\le 2\alpha$.
 \end{lemma}
\begin{proof}
  If $p_i\notin S$, then $[V(2^{p_i})]=2$ by assumption, so
  \[[V^>(\alpha)]=\prod_{i=1}^t[V(2^{p_i})]=2^t\le 2^{\log(\alpha)+1}\le 2\alpha.\qedhere\]
\end{proof}

\begin{lemma}\label{lem:3.2}
  Let $\alpha$ be a natural number, with binary decomposition $\alpha
  = 2^{p_1} + \cdots +2^{p_{t}}$. Suppose that there is $n\in Y$ such
  that $p_i\in\{n-e(n)-1,\dots,n-1\}$ for all $i=1,\dots,t$. Then
  $[V^>(\alpha)]\le 2^{10dn}$.  More precisely,
  $[V^>(\alpha)]=2^{\alpha/2^{n-e(n)-1}}$.
\end{lemma}
\begin{proof}
  Recall that we defined $e(n)=\lfloor \log(5dn)\rfloor$,
  see~\eqref{def:e}, and that, by Theorem~\ref{8props}(2), we have
  $[V(2^i)]=2^{2^{i-(n-e(n)-1)}}$ for all
  $i\in\{n-e(n)-1,\dots,n-1\}$. Then
  \begin{align*}
    \log[V^>(\alpha)]&=\log\prod_{i=1}^t[V(2^{p_i})]=\log\prod_{i=1}^t2^{2^{p_i-(n-e(n)-1)}}\\
    &= \sum_{i=1}^t2^{p_i-(n-e(n)-1)}=\frac{\alpha}{2^{n-e(n)-1}}\\
    &\le 2^{e(n)+1}\le 10dn.\qedhere
  \end{align*}
\end{proof}

\begin{proposition}\label{prop:bdV>}
  Let $\alpha$ be a natural number.  Then $[V^>(\alpha)]<2\alpha^{22d}$.
\end{proposition}
\begin{proof}
  Write $\alpha = 2^{p_1} + \cdots + 2^{p_{t}}$ in binary. Write again
  $S_n=\{n-e(n)-1,\dots,n-1\}$. For all $n\in\N$, set
  $\alpha_n=\sum_{p_i\in S_n}2^{p_i}$. Set $\gamma=\sum_n\alpha_n$ and
  $\delta=\sum_{p_i\notin S}2^{p_i}$, so that
  $\alpha=\gamma+\delta$. By definition of the sets $V^>(m)$, we have
  $[V^>(\alpha)]=[V^>(\gamma)][V^>(\delta)]$.  By Lemma~\ref{lem:3.2}
  we have $[V^>(\alpha_n)] \le 2^{10dn}$ for all $n$.

  Note now that, by~\eqref{gosha:3}, if
  $n<n'\in Y$ then $500n<n'$. Let $m\in\N$ be maximal such that $\alpha_m\neq0$. We deduce
  \[[V^>(\gamma)]= \prod_{n\le m, n\in Y}[V^>(\alpha_n)]<
  \prod_{i\in\N}2^{10md/500^i}\le 2^{10md\,500/499}.
  \]
  Moreover, from the binary form of $\alpha$, we get $\alpha \geq
  2^{m-e(m)-1}$. Recall that $e(m)=\lfloor \log(5dm)\rfloor$, hence
  by~\eqref{gosha:1} we have $m-e(m)-1>\frac m2$, so
  $[V^>(\gamma)]\le \alpha^{21d}$.

  Finally, by Lemma~\ref{lem:3.1}, we have $[V^>(\delta)]\le
  2\alpha$. Putting everything together, we get
  $[V^>(\alpha)]<2\alpha^{22d}$.
\end{proof}

\begin{lemma}\label{lem:3.4}
  Let $\alpha,\beta$ be natural numbers such that $\alpha+\beta \le
  2^{n-1}+2^{n-2}$ for some $n\in Y$. Then
  \[[V^<(\alpha)][V^>(\beta)]\le \frac1{2^{(n+1)(d+2)+2}}
  [V(2^{n-1})]^2.\]
\end{lemma}
\begin{proof}
  Write $\alpha = 2^{p_1} + \cdots +2^{p_t}$ in binary. Write again
  $S_m=\{m-e(m)-1,\dots,m-1\}$ and $\alpha_m=\sum_{p_i\in
    S_m}2^{p_i}$. Set now $\gamma=\sum_{m<n}\alpha_m$ and
  $\delta=\sum_{p_i\notin S}2^{p_i}$; we get
  $\alpha=\gamma+\delta+\alpha_n$, and by definition of the sets
  $V^>(n)$ we get $[V^>(\alpha)]=[V^>(\gamma)][V^>(\delta)][V^>(\alpha_n)]$.

  \noindent As in Proposition~\ref{prop:bdV>},
  \[[V^>(\gamma)]= \prod_{m<n/500, m\in Y}[V^>(\alpha_m)] < \prod_{i\in\N}2^{10dn/500^{i+1}}\le 2^{10dn/499}\le 2^{dn/48} .\]
  By Lemma~\ref{lem:3.1}, we get
  \[[V^>(\delta)]\le 2\delta\le 2\alpha.\]
  By Lemma~\ref{lem:3.2}, we get
  \[[V^>(\alpha_n)]=2^{\alpha_n / 2^{n-e(n)-1}}\le 2^{\alpha/2^{n-e(n)-1}}.\]
  Therefore,
  \[[V^>(\alpha)]\le 2\alpha 2^{dn/48}2^{\alpha / 2^{n-e(n)-1}}.\]
  By the definition of sets $V^<$ and $V^>$, we get
  $[V^<(\alpha)]=[V^>(\alpha)]$, so
  \[[V^<(\alpha)][V^>(\beta)]\le 4(\alpha\beta) 2^{dn/24}2^{\frac{\alpha+\beta}{2^{n-e(n)-1}}}.\]
  Since $\alpha+\beta \le 2^{n-1}+2^{n-2}$ so $4\alpha\beta\le 2^{2n}$, we get
  \begin{align*}
    \log([V^<(\alpha)][V^>(\beta)]) &\le 2n+dn/24+\frac{2^{n-1}+2^{n-2}}{2^{n-e(n)-1}}\\
    &= 2n+dn/24+2^{e(n)}+2^{e(n)-1}.
    \intertext{By Theorem~\ref{8props}(2) we have $[V(2^{n-1})]=2^{2^{e(n)}}$, so}
    \log([V^>(\alpha)][V^<(\beta)]) &\le 2n+dn/24+\log([V(2^{n-1})]^2)-2^{e(n)-1}\\
    &\le 2n+dn/24+\log([V(2^{n-1})]^2)-5dn/4\\
    &\le \log([V(2^{n-1})]^2)-\big((n+1)(d+2)+2\big)
  \end{align*}
  as required. The last inequality holds thanks to
  our assumption $d\ge 50$. Indeed, since $n\in Y$ then by ~\eqref{gosha:1} we have $2^{n+1}\geq 18d$, and so $n\geq 9$, hence $(d/5-3)(n-5)\geq 20$, which implies the last inequality.
\end{proof}

\begin{lemma}\label{lem:3.5}
  Let $F(2^n),U(2^n),V(2^n),S$ be as in Theorem~\ref{8props}.  Let
  $\mathscr E$ be defined as in~\eqref{def:E}.  Then the algebra
  $A/\mathscr E$ has Gelfand-Kirillov dimension at most $45d$.
\end{lemma}
\begin{proof}
  By Proposition~\ref{prop:A/Egrowth}, we have $\dim(A(k)/\mathscr
  E(k))\le \sum_{j=0}^k[V^<(k-j)][V^>(j)].$ By Proposition~\ref{prop:bdV>}
  we have $\dim(A(k)/\mathscr E(k))\le \sum_{j=0}^k4k^{44d}\le
  8k^{45d}$. Therefore, $\GKdim(A/\mathscr E)\le 45d$.
\end{proof}

\section{Constructing $F(2^n)$ from $U(2^m),V(2^m)$ for $m<n$}\label{ss:F}
In this~\S, we construct the sets $F(2^n)\subseteq A(2^n)$ that let us
apply Theorem~\ref{8props}. We proceed by steps:
\begin{lemma}\label{lem:F'}
  Let the notation be as in Theorem~\ref{8props} and
  Theorem~\ref{thm:gosha}.  Consider all $D(k)\subseteq A(k)$ with
  $2^n+2^{n-1}\le k\le 2^n+2^{n-1}+2^{n-2}$.  Suppose we defined sets
  $U(2^m)\subseteq A(2^m)$ for all $m<n$, and suppose $n\in Y$.

  Then there exists a linear $\K$-space $F'(2^n)\subseteq A(2^n)$ with the
  following properties:
  \begin{itemize}
  \item $0<\dim F'(2^n)\le \frac12\dim V(2^{n-1})^2$;
  \item for all $i,j\ge0$ with $i+j=k-2^n$ we have $D(k)\subseteq
    A(i)F'(2^n)A(j)+U^<(i)A(k-i)+A(k-j)U^>(j)$, with the sets
    $U^<(i),U^>(i)$ defined in~(\ref{def:U<}--\ref{def:V>}).
  \end{itemize}
\end{lemma}
\begin{proof} By Lemma~\ref{lem:A=U+V}, we have $U^{<}(i)\oplus
  V^{<}(i)=A(i)$ and $U^{>}(j)\oplus V^{>}(j)=A(j)$. Therefore,
  $A(i)A(2^n)A(j)= (U^{<}(i)\oplus V^{<}(i))A(2^n)(U^{>}(j)\oplus
  V^{>}(j))$.  Consequently, $A(i+2^n+j)=V^{<}(i)A(2^n)V^{>}(j)\oplus
  (U^<(i)A(k-i)+A(k-j)U^>(j))$.

  Consider $f\in D(k)$.  We can write $f$ in the form $f=f'+f''$,
  with
  \[f'=\sum_{c\in V^<(i),d\in V^>(j)}cz_{c,d,f}d\in V^<(i)A(2^n)V^>(j)\quad\text{for some } z_{c,d,f}\in A(2^n)\]
  and $f''\in U^<(i)A(k-i)+A(k-j)U^>(j)$.

  Still for that given $f$, we restrict the $c,d$ above to belong to
  bases of $V^<(i)$ and $V^>(j)$ respectively, and let
  $T(i,j,f)\subseteq A(2^n)$ be the subspace spanned by all the
  $z_{c,d,f}$ above. We then have $\dim T(i,j,f)\le\dim V^<(i)\dim
  V^>(j)$. Observe also $f\in
  A(i)T(i,j,f)A(j)+U^<(i)A(k-i)+A(k-j)U^>(j)$, because
  $V^<(i)\subseteq A(i)$ and $V^>(j)\subseteq A(j)$. Define
  \[F'(2^n)=\sum_{k=2^n+2^{n-1}}^{2^n+2^{n-1}+2^{n-2}}\sum_{f\in D(k)}\sum_{i+j=k-2^n}T(i,j,f).\]
  We have $2^{n-1}\le i+j \le 2^{n-1}+2^{n-2}$, so Lemma~\ref{lem:3.4} gives
  \begin{align*}
    \dim F'(2^n) &\le 2^{2n-2}2^{(n+1)d}\sup_{2^{n-1}\le i+j\le 2^{n-1}+2^{n-2}} \dim
    V^<(i)\dim V^>(j)\\
    &\le\tfrac12\dim V(2^{n-1})^2.\qedhere
  \end{align*}
\end{proof}

\begin{lemma}\label{lem:GKdown1}
  Let $R$ be a commutative finitely generated graded algebra of
  Gelfand-Kirillov dimension $t$.  Let $\mathscr I$ be a principal
  homogeneous ideal in $R$, that is, an ideal generated by one
  homogeneous element. Then $R/\mathscr I$ has Gelfand-Kirillov
  dimension at least $t-1$.
\end{lemma}
\begin{proof}
  We write $\mathscr I=cR$ for some homogeneous $c\in R$, and set
  $Q=R/\mathscr I$. We write $R({\le} n)$ for the subspace of $R$
  consisting of elements of degree $\le n$, and define $\mathscr I({\le}
  n)$ and $Q({\le} n)$ similarly. Then, as $R$ and $Q$ graded, we have
  $Q({\le} n)=R({\le} n)/\mathscr I({\le} n)=R({\le} n)/cR({\le} n-\deg c)$;
  from $\dim cR({\le} n-\deg c)\le\dim R({\le} n-\deg c)$ we get $\dim
  Q({\le} n)\ge\dim R({\le} n)-\dim R({\le} n-\deg c)$.

  Suppose for contradiction $\GKdim Q<t-1$, so $\GKdim Q=t-1-\epsilon$
  for some $\epsilon>0$. Consider $q\in(t-1-\epsilon,t-1)$. Then, by
  the definition of Gelfand-Kirillov dimension, we have $\dim Q({\le}
  n)<n^q$ for almost all $n$, so there is $C\in\mathbb R$ such that
  $\dim Q({\le} n)<Cn^q$ for all $n\in\N$. Observe now, for all
  $k\in\N$, that
  \begin{align*}
    \dim R({\le} k\deg c) &= \sum_{i=1}^k\dim R({\le} i\deg c)-\dim R({\le}(i-1)\deg c)\\
    &\le \sum_{i=1}^k\dim Q({\le} i\deg c)<\sum_{i=1}^k C(i\deg c)^q<Ck(k\deg c)^q,
  \end{align*}
  so $\GKdim R\le q+1$, a contradiction with $q<t-1$.
\end{proof}

\begin{lemma}\label{lem:s1}
  Let $\K$ be an algebraically closed field, let $n$ be a natural
  number, and let $T\subseteq A(2^n)$ and $Q\subseteq A(2^{n+1})$ be
  $\K$-linear spaces such that $\dim T+4\dim Q\leq \dim A(2^n)-2$.

  Then there exists a $\K$-linear space $F\subseteq A(2^n)$ of
  dimension at most $\dim A(2^n)-2$ such that $T\subseteq F$ and
  $Q\subseteq FA(2^n)+A(2^n)F$.
\end{lemma}
\begin{proof}
  Choose a $\K$-linear complement $C\subseteq A(2^n)$ to $T$; we have
  \begin{equation}\label{eq:Csum}
    C\oplus T=A(2^n).
  \end{equation}
  Choose also a basis $\{c_1,\dots,c_s\}$ of $C$ with $s=\dim
  A(2^n)-\dim T$.

  Let $R=\K[y_1,\dots,y_s,z_1,\dots,z_s]$ be the ring of polynomials
  in $2s$ indeterminates, and let $Y,Z$ be two non-commuting
  indeterminates over $R$. Define a $\K$-linear map $\Phi\colon C\to
  RY+RZ$ by
  \[\Phi(c_t)=y_tY+z_tZ\qquad\text{ for }t=1,\dots,s.\]
  Using~\eqref{eq:Csum}, extend $\Phi$ to a $\K$-linear map $A(2^n)\to
  RY+RZ$ satisfying $\ker(\Phi)=T$.

  Consider now $f\in A(2^{n+1})$.
  Then $f\in CC+\ker(\Phi)A(2^n)+A(2^n)\ker(\Phi)$.  It
  follows that there are $\alpha^{t,u}_{f}\in\K$ such that
  \[f\equiv\sum_{1\le t,u\le s}\alpha^{t,u}_{f}c_tc_u\mod{\ker(\Phi)A(2^n)+A(2^n)\ker(\Phi)}.\]

  \noindent Define now a $\K$-linear map $\Psi\colon A(2^{n+1})\to
  R\langle Y,Z\rangle$ by
  \[\Psi(c_tc_u)=\Phi(c_t)\Phi(c_u),\qquad
  \ker(\Psi)=A(2^n)\ker(\Phi)+\ker(\Phi)A(2^n).\]
  We get
  $\Psi(f)=\Psi(\sum_{1\le t,u \le s}\alpha ^{t,u}_{f}c_tc_u)$
  and so
  \[\Psi(f)=p^{YY}_{f}YY+p^{YZ}_{f}YZ+p^{ZY}_{f}ZY+p^{ZZ}_{f}ZZ\]
  for some polynomials $p^{YY}_{f},\dots,p^{ZZ}_{f}\in R$. Define a $\K$-linear subspace
  $E$ of $R$ as follows:
  \[E=\sum_{f\in Q}\K p^{YY}_{f}+\K p^{YZ}_{f}+\K p^{ZY}_{f}+\K
  p^{ZZ}_{f}.
  \]
  Observe that $\dim E\leq 4 \dim Q$, and hence $\dim T+\dim E\leq
  \dim A(2^n)-2$.

  Given $r\in R$ and $\boldsymbol\eta =
  (\eta_1,\dots,\eta_s),\boldsymbol\zeta=(\zeta_1,\dots,\zeta_s)\in\K^s$,
  we denote by $r(\boldsymbol\eta,\boldsymbol\zeta)$ the image in $\K$
  of $r$ after substituting $y_t:=\eta_t$ and $z_t:=\zeta_t$, for all
  $t=1,\dots,s$.  We will show that there are
  $\boldsymbol\eta,\boldsymbol\zeta\in\K^s$ such that
  $g(\boldsymbol\eta,\boldsymbol\zeta)=0$ for all $g\in E$; namely, if
  we substitute $y_t:=\eta_t$ and $z_t:=\zeta_t$ for all $t=1,\dots,s$
  in $g$, we get $0$.  Moreover, we will find $u,v\in\{1,\dots,s\}$
  such that $\Phi(c_u)(\boldsymbol\eta,\boldsymbol\zeta)$ and
  $\Phi(c_v)(\boldsymbol\eta,\boldsymbol\zeta)\in \K Y+\K Z$ are linearly
  independent over $\K$.

  We proceed by contradiction.  Assume that all assignments
  $y_t:=\eta_t$, $z_t:=\zeta_t$ satisfying
  $E(\boldsymbol\eta,\boldsymbol\zeta)=0$ also satisfy
  $\eta_u\zeta_v-\zeta_u\eta_v=0$ for all $u, v\in\{1,\dots,s\}$. The
  polynomials $y_tz_u-z_ty_u$ vanish on the zero-set of $E$, so by
  Hilbert's Nullstellensatz there is $m\in\N$ such that
  $(y_tz_u-z_ty_u)^m\in RE$.  It follows that $R/RE$ has
  Gelfand-Kirillov dimension at most $s+1$: it is a finite dimensional
  module over
  $\sum_{X\subset\{y_1,\dots,y_s,z_1,\dots,z_s\}\colon\#X=s+1}\K[X]$.

  On the other hand, by applying $\dim(E)$ times
  Lemma~\ref{lem:GKdown1}, we see that the dimension of $R/RE$ is at
  least $2s-\dim E$. Since $\dim(E)\le s-2$, we have reached a
  contradiction. It follows that we can find
  $\boldsymbol\eta,\boldsymbol\zeta\in\K^s$ and indices
  $u,v\in\{1,\dots,s\}$ such that $\eta_u\zeta_v-\zeta_u\eta_v\neq0$
  and $E(\boldsymbol\eta,\boldsymbol\zeta)=0$ .

  Define now a $\K$-linear mapping $\overline\Phi\colon C\to\K Y+\K Z$
  by $\overline\Phi(c_t)=\eta_tY+\zeta_tZ$ for $t=1,\dots,s$. Using
  $C\oplus T=A(2^n)$, extend it to a mapping $\overline\Phi\colon
  A(2^n)\to\K Y+\K Z$ by the condition $T\subseteq\ker
  \overline\Phi$. Then $\overline\Phi(c_u)=\eta_uY+\zeta_uZ$ and
  $\overline\Phi(c_v)=\eta_vY+\zeta_vZ$ give two elements that are
  linearly independent over $\K$. As before, define
  $\overline\Psi\colon A(2^{n+1})\to\K\langle Y,Z\rangle$ by
  \[\overline\Psi(c_tc_u)=\overline\Phi(c_t)\overline\Phi(c_u),\qquad
  \ker(\overline\Psi)=A(2^n)\ker(\overline\Phi)+\ker(\overline\Phi)A(2^n),\]
  and set
  \[F:=\ker\overline\Phi.
  \]
  By construction, we have $T\subseteq\ker\overline\Phi$ so
  $T\subseteq F$ as required. Because
  $\overline\Phi(c_u):=\eta_uY+\zeta_uZ$ and
  $\overline\Phi(c_v):=\eta_vY+\zeta_vZ$ are $\K$-linearly
  independent, we have $\dim F\leq \dim A(2^n)-2$ as
  required. Finally, from $E(\boldsymbol\eta,\boldsymbol\zeta)=0$ we
  get
  $p^{YY}_f(\boldsymbol\eta,\boldsymbol\zeta)=\cdots=p^{ZZ}_f(\boldsymbol\eta,\boldsymbol\zeta)=0$
  for all $f\in Q$, so $\overline\Psi(Q)=0$ and therefore $Q\subseteq
  \ker(\overline\Psi)=A(2^n)\ker(\overline\Phi)+\ker(\overline\Phi)A(2^n)=FA(2^n)+A(2^n)F$.
\end{proof}

\begin{lemma}\label{lem:s2}
  Suppose that sets $U(2^m),V(2^m)$ were already constructed for all
  $m<n$, and satisfy the conditions of Theorem~\ref{8props}. Let
  $D(k)$ be as in Theorem~\ref{thm:gosha}.  Define a $\K$-linear
  subspace $Q\subseteq A(2^{n+1})$ as follows:
  \[Q=\sum_{k=2^n+2^{n-1}}^{2^n+2^{n-1}+2^{n-2}}\sum_{f\in D(k)}\sum_{i+j=2^{n+1}-k}V^{>}(i)fV^{<}(j).\]
  Then $\dim Q\leq \frac14(\frac12\dim V(2^{n-1})^2-2)$.
\end{lemma}
\begin{proof}
  By Lemma~\ref{lem:3.4}, the inner sum has dimension at most $\dim
  V(2^{n-1})^2/2^{(n+1)(d+2)+2}$. Summing over all $i+j=2^{n+1}-k$
  multiplies by a factor of $2^{n+1}$ at most; summing over all $f\in
  D(k)$ multiplies by a factor of $\dim D(k)\le k^d\le 2^{(n+1)d}$ at
  most; and summing over all $k$ multiplies by a factor of $2^{n-1}$
  at most. Therefore, $\dim Q \leq 2^{n-1}2^{(n+1)d}2^{n+1}\frac{\dim
    V(2^{n-1})^2}{2^{(n+1)(d+2)+2}} \leq \frac1{16}\dim V(2^{n-1})^2$.

  Observe now that $\frac14\dim V(2^{n-1})^2\leq \frac12\dim V(2^{n-1})^2-2$,
  because $\dim V(2^{n-1})=2^{2^{e(n)}}\geq2^{2^1}\geq4$.  We get
  $\dim Q \leq\frac14(\frac12\dim V(2^{n-1})^2-2)$ as required.
\end{proof}

We are now ready to construct the space $F(2^n)$. Assume
$U(2^m),V(2^m)$ were already constructed for all $m<n$, and satisfy
the conditions of Theorem~\ref{8props}.
\begin{proposition}\label{prop:F}
  Let $\K$ be an algebraically closed field. With notation as in Lemma~\ref{lem:F'}, there is a linear $\K$-space
  $F(2^n)\subseteq A(2^n)$ satisfying $\dim
  F(2^n)\le \dim V(2^{n-1})^2-2$ and $F'(2^{n})\subseteq F(2^{n})+U(2^{n-1})A(2^{n-1})+A(2^{n-1})U(2^{n-1})$. Moreover, for all
  $k\in\{2^n+2^{n-1},\dots,2^n+2^{n-1}+2^{n-2}\}$ we have
  \begin{multline*}
    AD(k)A\cap A(2^{n+1})\subseteq A(2^n)F(2^n)+F(2^n)A(2^n)\\
    +A(2^{n-1})U(2^{n-1})A(2^n)+A(2^n)U(2^{n-1})A(2^{n-1})\\
    +U(2^{n-1})A(2^n+2^{n-1})+A(2^n+2^{n-1})U(2^{n-1}).
  \end{multline*}
\end{proposition}
\begin{proof}
  Consider the space $Q\subseteq A(2^{n+1})$ defined in
  Lemma~\ref{lem:s2}, and the space
  $T:=F'(2^n)+U(2^{n-1})A(2^{n-1})+A(2^{n-1})U(2^{n-1})\subseteq
  A(2^n)$, with $F'(2^n)$ as in Lemma~\ref{lem:F'}. We have $4\dim
  Q\le\frac12\dim V(2^{n-1})^2-2$ by Lemma~\ref{lem:s2}, and, using
  $A(2^n)=V(2^{n-1})^2\oplus(U(2^{n-1})A(2^{n-1})+A(2^{n-1})U(2^{n-1}))$,
  \begin{align*}
    \dim T &\leq \dim F'(2^n)+(\dim A(2^n)-\dim V(2^{n-1})^2)\\
    &\leq \dim A(2^n)-\tfrac12\dim V(2^{n-1})^2.
  \end{align*}
  Therefore, $\dim T+4\dim Q\leq \dim A(2^n)-2$ and we may apply
  Lemma~\ref{lem:s1} to obtain a set $F$.

  Consider $i,j,k\in\N$ with $i+j+k=2^{n+1}$, and consider $f\in
  D(k)$. By Lemma~\ref{lem:A=U+V}, we have $U^>(i)+V^>(i)=A(i)$ and
  $U^<(j)+V^<(j)=A(j)$. Therefore,
  \begin{align*}
    A(i)fA(j) &= (U^>(i)+V^>(i))f(U^<(j)+V^<(j))\\
    &\subseteq V^>(i)fV^<(j) + U^>(i)A(2^{n+1}-i)+A(2^{n+1}-j)U^<(j),
    \intertext{so}
    A(i)D(k)A(j) &\cap A(2^{n+1})\subseteq Q + U^>(i)A(2^{n+1}-i)+ A(2^{n+1}-j)U^<(j).
  \end{align*}
  By assumption on $k$, we have $i+j\leq 2^{n-1}$, so
  Lemma~\ref{lem:A=U+V} yields
  $U^>(i)A(2^{n+1}-j)=(U^>(i)A(2^{n-1}-i))A(2^n+2^{n-1})\subseteq
  U(2^{n-1})A(2^n+2^{n-1})$, and similarly
  $A(2^n+2^{n-1})U^<(j)\subseteq A(2^n+2^{n-1})U(2^{n-1})$.

  Then, Lemma~\ref{lem:s1} gives $Q\subseteq
  A(2^n)F+FA(2^n)$, so
  \begin{align}
    AD(k)A \cap A(2^{n+1})&\subseteq Q + \sum_{i+j=2^{n+1}-k}U^>(i)A(2^{n+1}-i)+
    A(2^{n+1}-j)U^<(j)\notag\\
    & \subseteq A(2^n)F+FA(2^n)\label{eq:AFA2}\\
    & +A(2^{n-1})U(2^{n-1})A(2^n)+A(2^n)U(2^{n-1})A(2^{n-1})\notag\\
    & +U(2^{n-1})A(2^n+2^{n-1})+A(2^n+2^{n-1})U(2^{n-1}).\notag
  \end{align}

  Recall $U(2^{n-1})A(2^{n-1})+A(2^{n-1})U(2^{n-1})\subseteq
  T\subseteq F$.  Let $F(2^{n})\subseteq F$ be a linear $\K$-space
  satisfying $F(2^{n})\oplus
  (U(2^{n-1})A(2^{n-1})+A(2^{n-1})U(2^{n-1}))=F$. The last
  claim of the theorem holds when we substitute this equation
  into~\eqref{eq:AFA2}.

  Observe next that we have $\dim F(2^{n})= \dim F-\dim
  U(2^{n-1})A(2^{n-1})+A(2^{n-1})U(2^{n-1})$, so $\dim F(2^{n})\leq
  \dim A(2^n)-2-(\dim A(2^n)-\dim V(2^{n-1})^2)\leq \dim V(2^{n-1})^2
  -2$, and the first claim of our theorem holds.  Since
  $F'(2^{n})\subseteq F=F(2^{n})+U(2^{n-1})A(2^{n-1})+A(2^{n-1})U(2^{n-1})$, the proof is finished.
\end{proof}

\section{Proof of Theorems~\ref{thm:gosha} and~\ref{cor:gosha}}\label{ss:proof}

We are now ready to prove our main result, which implies
Theorem~\ref{thm:gosha}. By~\cite{lenagan-s:nillie}*{Theorem~5}, the
set $\mathscr E$ defined in~\eqref{def:E} is an ideal in $A$.
\begin{theorem}\label{AbaroverEnil}
  The algebra $A/\mathscr E$ is an algebra of Gelfand-Kirillov
  dimension at most $45d$, which is infinite dimensional over $\K$,
  and in which the image of $D(k)$ is zero for all $k$.
\end{theorem}
\begin{proof}
  We will apply inductively Theorem~\ref{8props} and
  Proposition~\ref{prop:F}. We start the induction with
  $U(2^0)=F(2^0)=0$ and $V(2^0)=\K x+\K y$. Assume now that we
  constructed $U(2^m),V(2^m)$ for all $m<n$. If $n\in Y$, we construct
  $F(2^n)$ using Proposition~\ref{prop:F}, while if $n\notin Y$, then
  we set $F(2^n)=0$. We then construct $U(2^n)$, $V(2^n)$ using
  Theorem~\ref{8props}.

  Consider now $k\in\N$ with $2^n<k<2^{n+1}$. We claim that $D(k)$ is
  contained in $\mathscr E$; to see that, it suffices to check
  $A(i)D(k)A(j)\subset T:=A(2^{n+1})U(2^{n+1})+U(2^{n+1})A(2^{n+1})$
  for all $i,j\in\N$ with $i+j+k=2^{n+2}$.

  Notice first that by Theorem~\ref{8props},
  $U(2^{n-1})A(2^{n-1})+U(2^{n-1})A(2^{n-1})\subseteq U(2^{n})$ and
  $F(2^{n})\subseteq U(2^{n})$. By Proposition~\ref{prop:F}, we get
  $F'(2^{n})\subseteq U(2^{n})$, and combining with
  Condition~\ref{prop:6} gives $A(2^n)F(2^n)+F(2^n)A(2^n)\subseteq
  U(2^{n+1})$ and $A(2^n)F'(2^n)+F'(2^n)A(2^n)\subseteq U(2^{n+1})$.

  If $i\ge2^{n+1}$, we may apply Proposition~\ref{prop:F} to get
  $A(i-2^{n+1})D(k)A(j)\subseteq A(2^n)F(2^n)+F(2^n)A(2^n)+U(2^{n+1})\subseteq U(2^{n+1})$, so
  $A(i)D(k)A(j)\subseteq T$. Similarly, if $j\ge2^{n+1}$, we get
  $A(i)D(k)A(j-2^{n+1})\subseteq U(2^{n+1})$ so $A(i)D(k)A(j)\subseteq T$. If
  $i,j\ge2^n$ then $A(i-2^n)D(k)A(j-2^n)\subseteq A(2^n)U(2^n)+U(2^n)A(2^n)$
  so $A(i)D(k)A(j)\subseteq T$.

  If $i<2^n$ and $j<2^{n+1}$, then $D(k)\subseteq
  A(2^n-i)F'(2^n)A(2^{n+1}-j)+U^<(2^n-i)A(2^{n+1}+2^n-j)+A(2^{n+1}-i)U^>(2^{n+1}-j)$
  by Lemma~\ref{lem:F'}, so $A(i)D(k)A(j)\subseteq
  A(2^n)F'(2^n)A(2^{n+1})+U(2^{n})A(2^{n+1}+2^{n})+A(2^{n+1}+2^{n})U(2^{n})\subseteq T$. The case $i<2^{n+1},j<2^n$ is
  handled similarly.

  We may now conclude that $D(k)=0$ holds in $A/\mathscr E$.  By
  Lemma~\ref{lem:3.5}, the Gelfand-Kirillov dimension of $A/\mathscr
  E(n)$ is at most $45d$.

  Finally, we show that $A/\mathscr E$ is infinite dimensional over
  $\K$, as in~\cite{lenagan-s:nillie}*{Theorems~14,15}. Suppose, by
  contradiction, that $A/\mathscr E$ is finite-dimensional.  Since
  $A/\mathscr E$ is graded, we have $V(2^n)\subseteq\mathscr E$ for
  some $n\in\N$.  By definition of $\mathscr E$, we then have
  $V(2^n)^4\subseteq U(2^{n+2})$. Recall that $V(2^{n+2})\subseteq
  V(2^{n})^{4}$ by Condition~\eqref{prop:7}, hence
  $V(2^{n+2})\subseteq U(2^{n+2})$, a contradiction, because
  $V(2^{n+2})\cap U(2^{n+2})=0$ by Condition~\eqref{prop:5}.
\end{proof}

\subsection{Proof of Theorem~\ref{cor:gosha}}
We now describe the changes to be made to the argument above to prove
Theorem~\ref{cor:gosha}.  Assume that almost all $D(k)$ are zero,
namely that $D(k)=0$ for all $k\ge 2^t$.  We first construct the sets
$U(2^n)$, $F(2^n)$, $V(2^n)$ for all $n\leq t$, in the same way as in
the first part of the proof above, using Theorem~\ref{8props} and
Proposition~\ref{prop:F}.

Now, from the construction in Proposition~\ref{prop:F}, we have
$F(2^n)=0$ for all $n>t$, so we may add an assumption to the
construction of $U(2^n),V(2^n)$ that $\dim V(2^m)=1$ for all $m>t$. To
see how it can be done, assume $V(2^t)=\K m_1+\K m_2$ for some
$m_1,m_2\in A(2^n)$. Then define
\[V(2^{t+1})=\K m_1m_1\text{ and }U(2^{t+1})=(U(2^t)+\K m_2)A(2^t)+ A(2^{t})(U(2^t)+\K m_2),\]
and for all $n>t$ define
\[V(2^{n+1})=V(2^n)V(2^n)\text{ and }U(2^{n+1})=U(2^n)A(2^n)+A(2^n)U(2^n).
\]
We construct the sets $U^<(k),\dots,V^>(k)$ as
in~(\ref{def:U<}--\ref{def:V>}), and the ideal $\mathscr E$ as
in~\eqref{def:E}.

We will now show that $A/\mathscr E$ has at most quadratic growth.  By
Proposition~\ref{prop:A/Egrowth}, we have $\dim A(k)/ \mathscr
E(k)\leq \sum_{j=0}^k\dim V^<(k-j)\dim V^>(j)$. Observe that there is
$C\in\N$ such that $\dim V^{>}(k), \dim V^{<}(k)<C$ for all $k\in\N$,
because $\dim V(2^n)=1$ for $n>t$.  Hence $\dim A(k)/\mathscr
E(k)\leq (k+1)C^2$. We conclude that $A/\mathscr E$ has at most
quadratic growth.
 The proof that $A/\mathscr E$ is infinite dimensional is the same as in
  Theorem~\ref{cor:gosha}.
\section{Growth of algebras}\label{ss:arbitrary}
We prove Theorem~\ref{thm:arbitrary} in this~\S. First, we write
$d=f(1)$, and note that $f(n)\le d^n$ follows from
submultiplicativity. We will construct a $d$-generated monomial
algebra $B$ with growth approximately $f$, as a quotient of the free
algebra $A=\K\langle x_1,\dots,x_d\rangle$.

Let $M(n)$ denote the set of monomials in $A$ of degree $n$, and set
$M=\bigcup_{n\ge0}M(n)$. We call elements of $M$ alternatively
\emph{monomials} or \emph{words}. We construct subsets $W(2^n)$ of
monomials in $M(2^n)$, inductively as follows. Firstly,
$M(1)=W(1)=\{x_1,\dots,x_d\}$.  Assuming $W(2^{n-1})$ has been
constructed, let $C(2^n)$ be an arbitrary subset of $W(2^n)$ of
cardinality $\lceil f(2^{n+1})/f(2^n)\rceil$. Define then
$W(2^{n+1})=C(2^n)W(2^n)$. Set $W=\bigcup_{n\ge0}W(2^n)$. Finally, let
\[B=A/\langle w\in M\mid AwA\cap W=\emptyset\rangle\] be the monomial
algebra with relators all words that are not subwords of some word in
$W$.

Since $B$ is a monomial algebra, its growth is computed by estimating
the number of non-zero monomials of given length in $B$. We do this
at powers of $2$.

\begin{lemma}\label{growth:lower}
  The set $W$ is linearly independent in $B$.
\end{lemma}
\begin{proof}
  In a monomial algebra, monomials are linearly independent as soon as
  they are distinct and nonzero. If $w\in W$ were $0$ in $B$, we would
  have $w=avb$ for some $v\in M$ such that $AvA\cap W=\emptyset$; this
  contradicts $w\in W$.
\end{proof}

\begin{lemma}\label{growth:recur}
  Let $w\in M$ be a word of degree $2^m$. Assume that $w$ is a subword
  of $C(2^n)W(2^n)$ or of $W(2^n)C(2^n)$ for some $n>m$. Then $w$ is a
  subword of $C(2^{n-1})W(2^{n-1})$ or of $W(2^{n-1})C(2^{n-1})$.
\end{lemma}
\begin{proof}
  Let $w$ be a subword of some word $u\in W(2^n)C(2^n)\cup
  C(2^n)W(2^n)$; write $u=u_1u_2$ with $u_1,u_2\in W(2^n)$, and either
  $u_1\in C(2^n)$ or $u_2\in C(2^n)$. If $w$ is a subword of $u_1$ or
  of $u_2$, then $w$ is a subword of a word in
  $W(2^n)=C(2^{n-1})W(2^{n-1})$, so we are done.

  If $w$ overlaps $u_1$ and $u_2$, write $u_1=u_{11}u_{12}$ and
  $u_2=u_{21}u_{22}$ with $u_{11},\dots,u_{22}\in W(2^{n-1})$; then
  $u_{21}\in C(2^{n-1})$ because $u_2\in W(2^n)$. By assumption,
  $n-1\ge m$, so $u$ is a subword of $u_{12}u_{21}$, which belongs to
  $W(2^{n-1})C(2^{n-1})$ as required.
\end{proof}

\begin{lemma}\label{growth:upper}
  Every non-zero degree-$2^m$ monomial in $B$ is a subword of a
  monomial in $W(2^m)C(2^m)\cup C(2^m)W(2^m)$.
\end{lemma}
\begin{proof}
  Let $w\in M(2^m)$ be non-zero; so $awb\in
  W(2^n)=C(2^{n-1})W(2^{n-1})$ for some $n\ge m$. Apply then $m-n-1$
  times Lemma~\ref{growth:recur}.
\end{proof}

\begin{lemma}\label{growth:W}
  For all $n\in\N$, we have
  \[f(2^n)\le\#W(2^n)<2^nf(2^n).\]
\end{lemma}
\begin{proof}
  By induction; $\#W(1)=f(1)$, and $f(2^{n+1})\le f(2^n)\#C(2^n)<f(2^{n+1})+f(2^n)$, so
  \[f(2^{n+1})\le \#W(2^{n+1})=\#W(2^n)\#C(2^n)<2^n(f(2^{n+1})+f(2^n))\le 2^{n+1}f(2^{n+1}).\qedhere\]
\end{proof}

\begin{proof}[Proof of Theorem~\ref{thm:arbitrary}]
  By Lemmata~\ref{growth:lower} and~\ref{growth:W}, we have
  \[\dim B(2^n)\ge\#W(2^n)\ge f(2^n).\]

  The other inequality of Lemma~\ref{growth:W}, combined with
  $\#C(2^n)=\lceil f(2^{n+1})/f(2^n)\rceil$, implies
  \[\#C(2^n)W(2^n)\le\left\lceil\frac{f(2^{n+1})}{f(2^n)}\right\rceil2^nf(2^n)\le2^n(f(2^{n+1})+f(2^n));\]
  and similarly for $W(2^n)C(2^n)$. Each of
  these monomials has at most $2^n+1$ distinct subwords of length
  $2^n$. Therefore, by Lemma~\ref{growth:upper},
  \begin{align*}
    \dim B(2^n)&\le2(2^n+1)\#W(2^n)\#C(2^n)<2^{n+1}(2^n+1)(f(2^{n+1})+f(2^n))\\
    &\le2^{2n+3}f(2^{n+1}).\qedhere
  \end{align*}
\end{proof}

\begin{lemma}\label{lemma:2^n}
  If $f,g$ be two increasing functions such that $f(2^n)\le g(2^n)$
  holds for all $n$, then $f\precsim g$.
\end{lemma}
\begin{proof}
  For any $m\in\N$, let $n\in\N$ be minimal such that $m\le 2^n$. We
  have $f(m)\le f(2^n)\le g(2^n)\le g(2m)$, so $f\precsim g$.
\end{proof}

\begin{proof}[Proof of Corollary~\ref{cor:many}]
  Let $f$ be a submultiplicative, increasing function with $f(Cn)\ge
  nf(n)$. Note that this implies $f(n)\sim nf(n)$, and more generally
  $f(n)\sim p(n)f(n)$ for any polynomial $p$. By
  Theorem~\ref{thm:arbitrary} and Lemma~\ref{lemma:2^n}, there exists
  an algebra $B$ with $\dim B(n)\sim f(n)$. Again using $f(n)\sim
  nf(n)$, the growth of $B$ satisfies $v(n)\sim f(n)$.
\end{proof}

\subsection*{Thanks} The first author would like to express her
thanks to the University of G\"ottingen for their generous hospitality
whilst she was in residence as Emmy Noether Professor in May
2011. Both authors are grateful to Professor Zelmanov for the
inspiration that~\cite{zelmanov:openpbalgebras} provided.

Both authors are also very grateful to the referees for their very
helpful comments.

\end{document}